\documentclass[11pt,reqno,a4paper]{amsart}
\usepackage{mathrsfs}

\usepackage{amssymb,amsmath,amsfonts,amsthm,mathtools,mathabx}
\usepackage{breqn}

\newtheorem{theorem}{Theorem}[section]
\newtheorem{lemma}[theorem]{Lemma}
\newtheorem{propos}[theorem]{Proposition}

\newtheorem{defin}[theorem]{Definition}
\newtheorem{rmk}[theorem]{Remark}

\numberwithin{equation}{section}

\allowdisplaybreaks

\def\mint
   {\mkern12mu\hbox{\vrule height4pt
   depth-3.2pt
    width5pt}%
   \mkern-16.5mu\int}

\DeclareMathOperator*{\Conv}{\mathop{\scalebox{1.5}{\raisebox{-0.2ex}{$\ast$}}}}

 \newcommand{\loc}{ \rm{loc}}
\newcommand{\da}{{\delta}_1}
\newcommand{\db}{{\delta}_2}

\newcommand{\La}{\mathcal{L}}

\newcommand{\var}{\varepsilon} 
\newcommand{\varp}{\varphi}

\newcommand{\om}{\Omega} 
\newcommand{\p}{\partial}

\newcommand{\la}{\lambda} 
 
\newcommand{\IR}{\mathbb{R}}

\newcommand{\irn}{\IR^N}
\newcommand{\ird}{\IR^d}
\newcommand{\wbar}{\widebar}

\newcommand{\x}{x_1}
\newcommand{\y}{x_2}

\newcommand{\ID}{\mathcal{D}}

\newcommand{\IF}{\mathcal{F}}
\newcommand{\IM}{\mathcal{M}}
\newcommand{\IV}{\mathcal{V}}

\newcommand{\IS}{\mathcal{S}}

\newcommand{\IU}{\mathcal{U}}

\newcommand{\pp}{{\mathfrak p}}
\newcommand{\qq}{{\mathfrak q}}
\newcommand{\rr}{{\mathfrak r}}
\newcommand{\mm}{{\mathfrak m}}



\begin{document}

\bibliographystyle{plain}

\title[Lagrangian flows for vector fields with anisotropic regularity]
{Lagrangian flows for vector fields \\ with anisotropic regularity}

\author[A.~Bohun]{Anna Bohun}
\address{Anna Bohun, Departement Mathematik und Informatik, Universit\"at Basel, 
Rheinsprung 21, CH-4051, Basel, Switzerland}
\email{anna.bohun@unibas.ch}

\author[F.~Bouchut]{Fran\c{c}ois Bouchut}
\address{Fran\c{c}ois Bouchut, Universit\'e Paris-Est, Laboratoire d'Analyse et de Math\'ematiques Appliqu\'ees (UMR 8050), CNRS, UPEM, UPEC, F-77454, Marne-la-Vall\'ee, France}
\email{francois.bouchut@u-pem.fr} 

\author[G.~Crippa]{Gianluca Crippa}
\address{Gianluca Crippa, Departement Mathematik und Informatik, Universit\"at Basel, 
Rheinsprung 21, CH-4051, Basel, Switzerland}
\email{gianluca.crippa@unibas.ch}

\begin{abstract}
We prove quantitative estimates for flows of vector fields subject to anisotropic regularity conditions: some derivatives of some components are (singular integrals of) measures, while the remaining derivatives are (singular integrals of) integrable functions. This is motivated by the regularity of the vector field in the Vlasov-Poisson equation with measure density. The proof exploits an anisotropic variant of the argument in \cite{crippade1,jhde} and suitable estimates for the difference quotients in such anisotropic context. In contrast to regularization methods, this approach gives quantitative estimates in terms of the given regularity bounds. From such estimates it is possible to recover the well posedness for the ordinary differential equation and for Lagrangian solutions to the continuity and transport equations.  
\end{abstract}

\maketitle

\section{Introduction}

\subsection{Ordinary differential equations with non smooth vector field}

Given a smooth vector field $b : [0,T] \times \IR^N \to \IR^N$, the flow of $b$ is the smooth map $X : [0,T] \times \IR^N \to \IR^N$ satisfying 
$$
\begin{cases} 
\displaystyle \frac{dX}{ds} (s,x) = b(s,X(s,x)) \,, \quad s \in [0,T] \,, \\ \\
X(0,x) = x \,.
\end{cases}
$$
In the last years much attention has been devoted to the study of flows associated to vector fields that are not smooth (in particular, less than Lipschitz in the space variable). In this context, the correct notion of flow is that of {\em regular Lagrangian flow}, loosely speaking an ``almost-everywhere flow which (almost) preserves the Lebesgue measure'' (see Definition~\ref{d:rlf} for the precise definition).

Existence, uniqueness and stability of the regular Lagrangian flow have been proven by DiPerna and Lions \cite{lions} for Sobolev vector fields, and by Ambrosio \cite{ambrosio} for vector fields with bounded variation, in both cases under suitable bounds on the divergence of the vector field. Both results make use of the connection with the well posedness of the continuity equation
$$
\partial_t u + {\rm div}\, (bu)=0 \,,
$$
which in turn is analyzed thanks to the theory of renormalized solutions. We address the interested reader to \cite{bologna,amcri-edi,tesi,D} for a detailed presentation of these results and for further references. 

\subsection{Quantitative estimates for the ordinary differential equation}

An alternative and more direct approach has been introduced in \cite{crippade1}. Many of the ODE results in \cite{lions} can be derived with simple a priori estimates, directly at Lagrangian level, by studying a functional measuring an ``integral logarithmic distance'' between flows.

In detail, given two regular Lagrangian flows $X$ and $\bar X$ associated to a vector field $b$, the idea is to consider the functional
\begin{equation}\label{e:functional}
\Phi_\delta(s) = \int \log \left( 1 + \frac{|X(s,x)-\bar X(s,x)|}{\delta} \right) \, dx \,,
\end{equation}
where $\delta>0$ is a given parameter (which will be optimized in the course of the proof) and the integration is performed on a suitable compact set.

It is immediate to derive the following lower estimate, for a given $\gamma>0$:
$$
\Phi_\delta(s) \geq \int_{\{ |X-\bar X| \geq \gamma \}} \log \left( 1 + \frac{\gamma}{\delta} \right) \, dx
= \La^N ( \{ |X-\bar X| \geq \gamma \} ) \log \left( 1 + \frac{\gamma}{\delta} \right) \,,
$$
that is, the measure of the superlevels of the difference between two regular Lagrangian flows is upper estimated by 
\begin{equation}\label{criterion}
\frac{\Phi_\delta(s)}{\log \left( 1 + \frac{\gamma}{\delta} \right)}\,.
\end{equation} 
A strategy for proving uniqueness if therefore deriving upper bounds on the functional $\Phi_\delta(s)$ which blow up in $\delta$ slower than $\log \left( 1 / \delta \right)$ as $\delta \to 0$.

Differentiating in time the functional and using the ordinary differential equation we obtain
\begin{equation}\label{differentiating}
\Phi'_\delta(s) \leq \int \frac{| b(X) - b(\bar X) |}{ \delta + | X - \bar X |} \, dx 
\leq \int \min \left\{ \frac{ 2 \| b \|_\infty}{\delta} \; ; \; \frac{| b(X) - b(\bar X) |}{ | X - \bar X |} \right\} \, dx \,.
\end{equation}
In \cite{crippade1} it has been noted that the estimate of the difference quotients in terms of the maximal function
$$
\frac{| b(X) - b(\bar X) |}{ | X - \bar X |} \lesssim  MDb (X) + MDb ( \bar{X} ) \,,
$$
together with the {\em strong} estimate for the maximal function \eqref{ee:strong}, imply an upper bound on $\Phi_\delta(s)$ independent of $\delta$. This allowed in \cite{crippade1} the proof of existence, uniqueness, stability (with an effective rate), compactness, and mild regularity for the regular Lagrangian flow associated to a vector field with Sobolev regularity $W^{1,p}$, with $p>1$. We note in passing that the rate obtained in these estimates has been recently proven to be sharp (see \cite{cras}).

\medskip

The case $p=1$ (and the more general case of vector fields with bounded variation) was left open in the above analysis due to the failure of the strong estimate \eqref{ee:strong}: only the weak estimate \eqref{ee:weak} is available for $p=1$. This case has been studied in \cite{jhde} exploiting interpolation techniques in weak Lebesgue spaces. The weak estimate on the second term in the minimum in \eqref{differentiating} is interpolated with the (degenerating in $\delta$) $L^\infty$ estimate on the first term in the minimum. This gives an upper bound of the form
\begin{equation}\label{logscale}
\Phi_\delta(s) \lesssim \| Db \|_{L^1} \, \log \left( \frac{1}{\delta} \right) \,.
\end{equation}
This estimate is on the critical scale discriminating uniqueness. Therefore we have to play with constants: up to an $L^2$-remainder, the $L^1$-norm of $Db$ can be assumed to be arbitrarily small (we exploit here {\em equi-integrability} bounds on $Db$). This allows to re-gain smallness in \eqref{criterion} (notice that the $L^2$ part can be treated as in \cite{crippade1}).

For this reason the analysis in \cite{jhde} is not able to address the case when $Db$ is a measure (i.e., the case of a vector field with bounded variation). On the other hand, by considering {\em smooth maximal functions} instead of classical ones, and by exploiting more sophisticated tools from harmonic analysis, the case in which $Db$ is a {\em singular integral} of an $L^1$ functional can be treated with the same strategy. This extends the case $b \in W^{1,1}$ and is relevant for some applications to nonlinear PDEs (see \cite{vp,euler}). Results of existence, uniqueness, stability (with an effective rate), and compactness follows as in \cite{crippade1}. We refer to \cite{amcri-edi} and to the introduction of \cite{jhde} for a more detailed presentation.

\subsection{A split case and the main result of the present paper}

As mentioned above, the analysis in \cite{jhde} is not able to include the case when $Db$ is a measure (or a singular integral of a measure). However, in situations originating from models in mathematical physics, the vector field is endowed with a particular structure, and just {\em some} of the derivatives are singular integrals of measures, while the remaining derivatives are more regular.

For instance, the Vlasov-Poisson equation
$$
\begin{cases}
\partial_t f + v \cdot \nabla_x f + E(t,x) \cdot \nabla_v f = 0 \,, \\
E(t,x) = - \gamma \nabla_x U(t,x) \,, \quad - \Delta_x U(t,x) = \rho(t,x) = \int f(t,x,v) \, dv 
\end{cases}
$$
is a (nonlinear) transport equation with vector field $b(t,x,v) = (v,E(t,x))$. If we look at the case when the space density $\rho$ is a measure, it turns out that $D_x E$ is a singular integral of a measure, while all other derivatives of the vector field enjoy better regularity. However, we are not able to consider the case of $f$ a measure in $x,v$, that has been studied in \cite{ZM,MMZ}, since the characteristics are defined only almost everywhere.

This motivates the setting of the present paper. We we write $\IR^N = \IR^{n_1} \times \IR^{n_2}$ with coordinates $x_1$ and $x_2$, and split analogously the vector field according to $b = (b_1,b_2)$. Roughly speaking, we consider the case in which $D_1 b_2$ is a singular integral (in $\IR^{n_1}$) of a measure, while $D_1 b_1$, $D_2 b_1$ and $D_2 b_2$ are singular integrals (in $\IR^{n_1}$) of integrable functions:
$$
Db  =
\left( \begin{matrix} 
   S \ast L^1 & S \ast L^1\\
   S \ast \IM & S \ast L^1 
\end{matrix} \right)
$$ 
(in fact our assumptions are slightly more general: see assumption {\bf (R2)} in Section~\ref{s:reg}). Compared to \cite{jhde}, we are able to consider a situation in which some entries of the differential matrix $Db$ are measures. (From a PDE point of view, related contexts have been considered in \cite{lebris,lerner}).

The idea, analogous to the anisotropic regularization of \cite{bouarma,ambrosio}, is to ``weight'' differently the two (groups of) directions, according to the different degrees of regularity. In our context, this can be done by considering, instead of \eqref{e:functional}, a functional depending on {\em two} parameters $\delta_1$ and $\delta_2$, with $\delta_1 \leq \delta_2$, namely
\begin{equation}\label{anfunc}
\Phi_{\delta_1,\delta_2}(s) = \int \log \left( 1 + \left| \left( \frac{|X_1(s,x)-\bar X_1(s,x)|}{\delta_1} \, , \, 
\frac{|X_2(s,x)-\bar X_2(s,x)|}{\delta_2} \right) \right|  \right) \, dx \,.
\end{equation}
Following the same strategy as before (estimate of the difference quotients and interpolation in the minimum in \eqref{differentiating}), we derive the following bound, which replaces \eqref{logscale} in this context:
$$
\Phi_{\delta_1,\delta_2}(s) 
\lesssim  
\left[ \frac{\delta_1}{\delta_2} \| D_1 b_2 \|_\IM + \frac{\delta_2}{\delta_1} \| D_2 b_1 \|_{L^1}
+ \| D_1 b_1 \|_{L^1} + \| D_2 b_2 \|_{L^1} \right] \log \left( \frac{1}{\delta_2} \right)  \,.
$$
We need to gain some ``smallness'' in criterion \eqref{criterion}. Observe that $\| D_2 b_1 \|_{L^1}$, $\| D_1 b_1 \|_{L^1}$ and $\| D_2 b_2 \|_{L^1}$ can be assumed to be small, by the same equi-integrability argument as in \cite{jhde}. This is however not the case for $\| D_1 b_2 \|_\IM$. But we can exploit the presence of the coefficient $\delta_1 / \delta_2$ multiplying this term: both $\delta_1$ and $\delta_2$ have to be sent to zero, but we can do this with $\delta_1 \ll \delta_2$.

One relevant technical point in the proof is the estimate for the anisotropic difference quotients showing up when differentiating \eqref{anfunc}. We need an estimate of the form:
\begin{equation}\label{anest}
|f(x)-f(y)| \lesssim  \left| \left( \frac{x_1-y_1}{\delta_1} \, , \, \frac{x_2-y_2}{\delta_2} \right) \right|
\Big[ U(x) + U(y) \Big] \,.
\end{equation}
This is complicated by the fact that, as in the classical case, one expects to use a maximal function in $x_1$ and $x_2$ in order to estimate the difference quotients, but however this would not match (in terms of persistence of cancellations) with the presence of a singular integral in the variable $x_1$ only. This is resolved in Section~\ref{s:differ} by the use of tensor products of maximal functions, and will result in the proof of \eqref{anest} together with a bound of the form
$$
\| U \| \leq \delta_1 \| D_1 f\| + \delta_2 \| D_2 f \| \,.
$$

This is the plan how to obtain the proof of our main Theorem~\ref{fund}, containing the fundamental estimate for the distance between two regular Lagrangian flows associated to vector fields under the regularity assumption {\bf (R2)}. As recalled in Section~\ref{s:fund} we obtain as a corollary of Theorem~\ref{fund} existence, uniqueness, stability (with an effective rate) and compactness for regular Lagrangian flows, and well posedness for Lagrangian solutions to the continuity and transport equations. Applications to the Vlasov-Poisson equation will be detailed in \cite{vp}. See also \cite{maria}, where similar arguments have been applied to the study of the Vlasov-Poisson equation, also exploiting the notion of {\em maximal regular flow}.

\subsection*{Acknowledgment} This research has been partially supported by the SNSF grants 140232 and 156112.

\section{Background material}

This section is devoted to recalling some classical definitions and results from harmonic analysis. Most of the results below are stated without proofs, for which we refer to \cite{stein}. The proofs of the more specific results and additional comments can be found in \cite{jhde}.

\subsection{Weak Lebesgue spaces and equi-integrability}

We will denote by $\La^d$ the $d$-dimen\-sional Lebesgue measure and by $B_r(x)$ the open ball or radius $r>0$ centered at $x \in \ird$, shortened to $B_r$ in case the center of the ball is the origin of $\ird$.

\begin{defin}
Let $u$ be a measurable function on $\Omega\subset \ird$. For $1\leq p<\infty$, we set
$$
|||u|||^p_{M^p(\Omega)}=\sup_{\lambda>0} \{\lambda^p \La^d (\{ x \in \Omega \; : \; |u(x)|>\lambda\})\}
$$
and define the weak Lebesgue space $M^p(\Omega)$ as the space consisting of all such measurable functions $u:\Omega\rightarrow \IR$ with $|||u|||_{M^p(\Omega)}<\infty$. For $p=\infty$, we set $M^\infty(\Omega)=L^\infty(\Omega)$.
\end{defin}

Let us remark that the quantity $|||\cdot|||^p_{M^p(\Omega)}$ is {\em not} a norm, therefore we have chosen the notation with the three vertical bars, different from the usual one for the norm.

The following lemma shows that we can interpolate $M^1$ and $M^p$, with $p>1$, obtaining a bound on the $L^1$ norm, which depends logarithmically on the $M^p$ norm.

\begin{lemma}[Interpolation]
\label{interpollemma}
Let $u: \om\mapsto [0,+\infty)$ be a nonnegative measurable function, where $\om\subset\IR^d$ has finite measure. Then for every $1<p<\infty$, we have the interpolation estimate
$$
||u||_{L^1(\om)}\leq\frac{p}{p-1}|||u|||_{M^1(\om)}\left[1+\log\left(\displaystyle\frac{|||u|||_{M^p(\om)}}{|||u|||_{M^1(\om)}}\La^d(\om)^{1-\frac{1}{p}}\right)\right] \,,
$$
and analogously for $p=\infty$
$$
||u||_{L^1(\om)}\leq |||u|||_{M^1(\om)}\left[1+\log\left(\displaystyle\frac{||u||_{L^\infty(\om)}}{|||u|||_{M^1(\om)}}\La^d(\om) \right)\right] \,.
$$
\end{lemma}

We also recall the classical definition of equi-integrability.

\begin{defin}[Equi-integrability] Let $\Omega$ be an open subset of $\ird$. We say that a bounded family $\{ \varphi_i \}_{i \in I} \subset L^1(\Omega)$ is equi-integrable if the following two conditions hold:
\begin{itemize}
\item[(i)] For any $\var > 0$ there exists a Borel set $A \subset \Omega$ with finite measure such that $\int_{\Omega \setminus A} |\varphi_i| \, dx \leq \var$ for any $i \in I$;
\item[(ii)] For any $\var > 0$ there exists $\delta > 0$ such that, for every Borel set $E \subset \Omega$ with with $\La^d(E) \leq \delta$, there holds $\int_E | \varphi_i | \, dx \leq \var$ for any $i \in I$.
\end{itemize}
\end{defin}

The Dunford-Pettis theorem ensures that a bounded family in $L^1(\Omega)$
is relatively compact for the weak $L^1$ topology if and only if it is equi-integrable. 
Also, a sequence $u_n\in L^1(\ird)$ converges to $u$ in $L^1(\ird)$
if and only if it is equi-integrable and $u_n$ converges to $u$ locally in measure.
The following lemma can be proved with elementary tools.

\begin{lemma}\label{equi}
Consider a family $\{\varphi_i\}_{i\in I} \subset L^1(\om)$ which is bounded in $L^1(\om)$ and fix $1<p\leq\infty$. Then this family is equi-integrable if and only if for every $\var>0$, there exists a constant $C_\var$ and a Borel set $A_\var\subset\om$ with finite measure such that for every $i\in I$ one can write
$$
\varphi_i = \varphi_i^1 + \varphi_i^2 \,,
$$
with
$$
\| \varphi_i^1 \|_{L^1(\Omega)} \leq \var \qquad \text{ and } \qquad
{\rm spt}\,(\varphi_i^2)\subset A_\var,\quad
\| \varphi_i^2 \|_{L^p(\Omega)} \leq C_\var \qquad \text{ for all $i \in I$.}
$$
\end{lemma}

\subsection{Singular integrals}

We briefly summarize the classical Calder\'on-Zygmund theory of singular integrals.

\begin{defin}\label{singker} We say that K is a {\em singular kernel} on $\ird$ if
\begin{enumerate}
\item $K\in\IS'(\ird)$ and $\hat{K}\in L^\infty(\ird)$, 
\item $K|_{\ird\setminus \{0\}}\in L^1_{\rm{loc}}(\ird\setminus \{0\})$ and there exists a constant $A\geq 0$ such that
$$\int_{|x|>2|y|}|K(x-y)-K(x)|dx\leq A$$ for every $y\in\ird$.
\end{enumerate}
\end{defin}

We now state a classical result that allows the extension of (the convolution with) a singular kernel to an operator on $L^p$ spaces.  

\begin{theorem}[Calder\'on-Zygmund]
Let $K$ be a singular kernel and define
$$
Su = K \Conv u \qquad \text{for $u\in L^2(\ird)$}
$$
in the sense of multiplication in the Fourier variable. Then for every $1<p<\infty$ we have the strong estimate
\begin{equation}\label{e:ssi}
\| Su \|_{L^p(\ird)} \leq C_{N,p}(A+||\hat{K}||_{L^\infty})||u||_{L^p(\ird)} \qquad  u\in L^p\cap L^2(\ird)\,,
\end{equation}
and for $p=1$ the weak estimate
\begin{equation}\label{e:wsi}
|||Su|||_{M^1(\ird)} \leq C_{N}(A+||\hat{K}||_{L^\infty})||u||_{L^1(\ird)} \qquad u\in L^1\cap L^2(\ird) \,.
\end{equation}
\end{theorem} 

In addition, the operator $S$ can be extended to the whole $L^p(\ird)$ for any $1<p<\infty$ with values in $L^p(\ird)$, still satisfying \eqref{e:ssi}. For $p=1$, the operator $S$ extends to the whole $L^1(\ird)$ to an operator $S^{M^1}$ with values in $M^1(\ird)$, still satisfying \eqref{e:wsi}. However, a function in $M^1(\ird)$ is in general not integrable, therefore it does not define a distribution. Notice that, for $u \in L^1(\ird)$, we can define a tempered distribution $S^D \in \IS'(\ird)$ by the formula
\begin{equation}\label{bracket}
\langle S^Du,\varphi\rangle=\langle u,\tilde{S}\varp\rangle \qquad
\text{for every $\varp\in\IS(\ird)$,}
\end{equation}
where $\tilde{S}$ is the singular integral operator associated to the kernel $\tilde{K}(x) = K(-x)$. The same holds for $u$ a finite measure in $\ird$. The two operators $S^{M^1}$ and $S^D$ are different and cannot be identified. Since $\IF:L^1(\ird) \to L^\infty(\ird)$ is bounded, and by definition we have $\hat{K}\in L^\infty(\ird)$, the definition in \eqref{bracket} is equivalent to the definition in Fourier variables 
$$
\widehat{S^Du}=\hat{K}\hat{u}\,.
$$ 
 
We also recall a particular class of singular kernels:
\begin{defin}
A kernel $K$ is a {\em singular kernel of fundamental type} in $\ird$ if the following properties hold:
\begin{enumerate}
\item $K|_{\IR^d\setminus\{0\}}\in C^1(\IR^d\setminus\{0\})$,
\item There exists a constant $C_0\geq 0$ such that
\begin{equation}\label{c0}
|K(x)|\leq \frac{C_0}{|x|^d} \qquad x\in \ird\setminus\{0\} \,,
\end{equation}
\item There exists a constant $C_1\geq 0$ such that
\begin{equation}\label{c1}
|\nabla K(x)|\leq \frac{C_1}{|x|^{d+1}} \qquad x\in \ird\setminus\{0\} \,,
\end{equation}
\item There exists a constant $A_1\geq 0$ such that 
\begin{equation}
\left|\int_{R_1<|x|<R_2}K(x) dx\right| \leq A_1 \qquad
\text{for every $0<R_1<R_2<\infty$.}
\end{equation}
\end{enumerate}
\end{defin}

In particular, these conditions are sufficient to extend the function defined on $\ird\setminus\{0\}$ to a singular kernel $K$ on $\ird$, unique up to addition of a multiple of a Dirac delta at the origin, and which satisfies the estimates in Definition~\ref{singker}.

\subsection{Maximal functions}

We now recall the classical maximal function.
\begin{defin} Let $u\in L^1_{\loc}(\IR^d)$. The maximal function of $u$ is defined as 
$$
Mu(x)=\sup_{\var>0}\mint_{B_\var(x)}|u(y)|dy
$$
for every $x\in\IR^d$. 
\end{defin}

The maximal function $Mu$ is finite almost everywhere for $u\in L^p(\IR^d)$, for every $1\leq p \leq \infty$. For every $1<p\leq\infty$ we have the strong estimate
\begin{equation}\label{ee:strong}
\|Mu\|_{L^p(\IR^d)} \leq C_{d,p}||u||_{L^p(\IR^d)}\,,
\end{equation}
with only the weak estimate for $p=1$
\begin{equation}\label{ee:weak}
|||Mu|||_{M^1(\IR^d)} \leq C_{d}||u||_{L^1(\IR^d)}\,.
\end{equation}

\subsection{The smooth maximal function and cancellations}

Given two singular kernels of fundamental type $K_1$ and $K_2$, with bounded and smooth Fourier transform, we consider the associated singular integral operators $S_1$ and $S_2$. The composition $S_2 \circ S_1$ is still a singular integral operator $S$, associated to a singular kernel $K$ characterized by $\widehat{K}=\widehat{K_2}\widehat{K_1}$. In general, composing two weak estimates (as in \eqref{e:wsi}) is not well defined. However, there are cancellations in the convolution $K_2 \Conv K_1$ (that is, in the composition of the two singular integral operators), which allow us to define $S_2\circ S_1$. A very important result is that we can compose a special class of maximal functions with a singular integral operator, yielding a composition operator that is bounded $L^1\rightarrow M^1$ and $L^2\rightarrow L^2$. 

We consider a maximal function that is ``smaller'' than the classical maximal function, in order to allow cancellations with the singular integral operator. Here the absolute value is outside the integral, instead of inside. The result after taking smooth averages is a maximal function that is ``smoother'' than the classical maximal function.

\begin{defin}[Smooth maximal function]\label{d:smoothmax}
Given a family of functions $\{\rho^\nu\}_\nu\subset L^\infty_c(\ird)$, for every function $u\in L^1_{\loc}(\ird)$ we define the $\left\{\rho^\nu\right\}$-maximal function of $u$ as 
$$
M_{\left\{\rho^\nu\right\}}(u)(x)=\sup_\nu\sup_{\var>0} \left|\int_{\ird} \rho^\nu_\var(x-y)u(y)dy\right|=\sup_\nu\sup_{\var>0}\left|(\rho^\nu_\var *u)(x)\right| ,
$$
where as usual
$$
\rho^\nu_\var(x)=\frac{1}{\var^d}\rho^\nu\left(\frac{x}{\var}\right) .
$$
In the case when $u$ is a distribution, we take a smooth family $\{\rho^\nu\}_\nu\subset C^\infty_c(\ird)$ and define in the distributional sense
$$
M_{\left\{\rho^\nu\right\}}(u)(x)=\sup_\nu\sup_{\var>0}\left|\langle u,\rho^\nu_\var(x-\cdot)  \rangle\right|.
$$
\end{defin}

The importance of this class of maximal functions is that it is possible to define the composition $M_{\left\{\rho^\nu\right\}}S$ with a singular integral operator, which is impossible with the usual maximal function. The following theorem has been proved in \cite{jhde}.

\begin{theorem}\label{cancel}
Let $K$ be a singular kernel of fundamental type and let $S$ be the associated singular integral operator. Let $\{\rho^\nu\}_\nu\subset L^\infty(\ird)$ be a family of kernels such that 
$$
{\rm spt} \, \rho^\nu\subset B_1 \quad \textrm{ and } \quad 
\|\rho^\nu\|_{L^1(\ird)}\leq Q_1 \qquad \text{ for every $\nu$.}
$$
Assume that for every $\var>0$ and every $\nu$ there holds 
$$
\bigl\| \bigl(\var^d K(\var\cdot) \bigl) \ast \rho^\nu \bigr\|_{C_b(\ird)}\leq Q_2
\qquad \text{for every $\var>0$ and for every $\nu$.} 
$$
Then we have the following estimates.
\begin{enumerate}
\item There exists a constant $C_d$, depending on the dimension $d$ only,
such that
$$
|||M_{\{\rho^\nu\}}(Su)|||_{M^1(\ird)}\leq C_d\Bigl(Q_2+Q_1(C_0+C_1+||\widehat{K}||_{\infty}\Bigl)||u||_{L^1(\ird)}
$$
for every $u\in L^1\cap L^2(\ird)$. If $\{\rho^\nu\}\subset C^\infty_c(\ird)$, and $u$ is a finite measure on $\ird$, then the same estimate holds, where $Su$ is defined as a distribution $S^Du$:
$$
|||M_{\{\rho^\nu\}}(Su)|||_{M^1(\ird)}\leq C_d\Bigl(Q_2+Q_1(C_0+C_1+||\widehat{K}||_{ \infty}\Bigl)||u||_{\IM(\ird)}\,.
$$
\item If $Q_3=\sup_\nu||\rho^\nu||_{L^\infty(\ird)}$ is finite, then there exists $C_d$, depending on the dimension $d$ only, such that
$$
\| M_{\{\rho^\nu\}}(Su)\|_{L^2(\ird)}\leq C_d Q_3 \| \widehat{K}\|_{\infty} \|u\|_{L^2(\ird)}
$$ 
for every $u\in L^2(\ird)$.
\end{enumerate}
\end{theorem} 
 
\section{Regular Lagrangian flows}

As mentioned in the Introduction, we will deal with flows of non-smooth vector fields. The adequate notion of flow in this context is that of {\em regular Lagrangian flow}. Given a vector field $b(s,x):(0,T)\times\irn\rightarrow\irn$, we assume the following growth condition:

\begin{itemize}
\item[{\bf(R1)}] The vector field $b(s,x)$ can be decomposed as
$$
\frac{b(s,x)}{1+|x|}=\tilde{b}_1(s,x)+\tilde{b}_2(s,x) \,,
$$
with
$$
\tilde{b}_1\in L^1((0,T);L^1(\irn)) \qquad \text{ and } \qquad  \tilde{b}_2\in L^1((0,T);L^\infty(\irn))\,.
$$
\end{itemize}

Given a vector field satisfying {\bf (R1)}, we codify in the following definition of {\em regular Lagrangian flow} the notion of ``almost everywhere flow which almost preserves the Lebesgue measure''.

\begin{defin}[Regular Lagrangian flow]\label{d:rlf}
If $b$ is a vector field satisfying {\bf (R1)}, then for fixed $t\in[0,T)$, a map $$X\in C([t,T]_s;L^0_{\loc}(\irn_x))\cap \mathcal{B}([t,T]_s;\log L_{\loc}(\irn_x))$$ is a regular Lagrangian flow in the renormalized sense relative to $b$ starting at $t$ if we have the following:
\begin{enumerate}
\item The equation 
$$\p_s\bigl(\beta(X(s,x))\bigl)=\beta ' (X(s,x))b(s,X(s,x))$$ holds in $\ID'((t,T)\times\irn)$, for every function $\beta\in C^1(\irn;\IR)$ that satisfies
$|\beta(z)|\leq C(1+\log(1+|z|))$ and $|\beta'(z)|\leq\frac{C}{1+|z|}$ for all $z\in\irn$,
\item $X(t,x)=x$ for $\La^N$-a.e $x\in\irn$, 
\item There exists a constant $L\geq 0$ such that $\int_{\irn} \varp(X(s,x))dx\leq L\int_{\irn} \varp(x)dx$ for all measurable $\varp:\irn\rightarrow[0,\infty)$.  
\end{enumerate}
\end{defin}

We will usually refer to the constant $L$ in Definition~\ref{d:rlf}(3) as the {\em compressibility constant} of the flow. We have denoted by $L^0_\loc$ the space of measurable functions endowed with the local convergence in measure, and by $\mathcal{B}$ the space of bounded functions. 

We define the sublevel of the flow as 
\begin{equation}\label{e:sub}
G_\la=\{x\in\irn: |X(s,x)|\leq\la \text{ for almost all } s\in[t,T]\} \,.
\end{equation}

The following lemma gives an estimate for the decay of the superlevels of a regular Lagrangian flow.
\begin{lemma}\label{estsuper}
Let $b:(0,T)\times\irn\rightarrow \irn$ be a vector field satisfying  {\bf (R1)} and let $X:[t,T]\times\irn\rightarrow\irn$ be a regular Lagrangian flow relative to $b$ starting at time $t$, with compressibility constant $L$. 
Then for all $r,\lambda>0$
$$
\La^N(B_r\setminus G_\lambda)\leq g(r,\lambda)\,,
$$
where the function $g$ depends only on $L$, $\|\Tilde{b}_1\|_{L^1((0,T);L^1(\irn))}$ and $\|\Tilde{b}_2\|_{L^1((0,T);L^\infty(\irn))}$ and satisfies $g(r,\lambda)\downarrow 0$ for $r$ fixed and $\lambda\uparrow \infty$. 
\end{lemma}

Indeed the regular Lagrangian flow $X$ has a logarithmic summability, and this clarifies the class of renormalization functions $\beta$ considered in Definition~\ref{d:rlf}(1). See \cite{jhde} for the proof.

\section{Regularity assumptions and the anisotropic functional}\label{s:reg}

We wish to consider a regularity setting of the vector field $b(t,x)$ in which the (weak) regularity has a different character with respect to different directions in space. We split $\irn$ as $\irn=\IR^{n_1}\times\IR^{n_2}$ with variables $\x\in\IR^{n_1}$ and $\y\in \IR^{n_2}$. We denote by $D_1=D_{x_1}$ the derivative with respect to the first $n_1$ variables $x_1$, and by $D_2=D_{x_2}$ the derivative with respect to the last $n_2$ variables $x_2$. Accordingly, we denote $b=(b_1,b_2)(s,x_1,x_2)$. For $X(s,\x,\y)$ a regular Lagrangian flow associated to $b$ we denote $X=(X_1,X_2)(s,x_1,x_2)$. 

We are going to assume that $D_1b_2$ is ``less regular'' than $D_1b_1,D_2b_1,D_2b_2$: the derivative $D_1b_2$ is a singular integral of a measure, whereas the other derivatives are singular integrals of $L^1$ functions. This is made precise as follows:

\begin{itemize}
\item[{\bf (R2)}] We assume that
\begin{equation}\label{eq:r2bmat}
Db  =
\left( \begin{matrix} 
   D_1 b_1 & D_2b_1\\
  D_1 b_2 & D_2b_2 
\end{matrix} \right)
=
\left( \begin{matrix} 
  \gamma^1 S^1 \pp &     \gamma^2 S^2\qq\\
   \gamma^3 S^3 \mm &    \gamma^4 S^4\rr
\end{matrix} \right)\,,
\end{equation}
where the sub-matrices have the representation
$$
\arraycolsep=1.4pt\def\arraystretch{2.2}
 \setlength\arraycolsep{15pt}
 \begin{array}{lcl} i,j\in\{1,\ldots,n_1\}: &   & i\in\{1,\ldots,n_1\},\quad j\in\{n_1+1,\ldots,n_2\}: \\
   (D_1b_1)^i_j=\displaystyle\sum^m_{k=1} \gamma_{jk}^{1i}(s,\y) S^{1i}_{jk}\pp^i_{jk}(s,\x)  &   &  (D_2b_1)^i_j=\displaystyle\sum^m_{k=1}  \gamma_{jk}^{2i}(s,\y) S^{2i}_{jk}\qq^i_{jk}(s,\x) \\
   i\in\{n_1+1,\ldots,n_2\}, \quad j\in\{1,\ldots,n_1\} : &&  i\in\{n_1+1,\ldots,n_2\}, \quad j\in\{n_1+1,\ldots,n_2\} :\\
   (D_1b_2)^i_j=\displaystyle\sum^m_{k=1} \gamma_{jk}^{3i}(s,\y)S^{3i}_{jk}\mm^i_{jk}(s,\x) & & (D_2b_2)^i_j=\displaystyle\sum^m_{k=1}  \gamma_{jk}^{4i}(s,\y) S^{4i}_{jk}\rr^i_{jk}(s,\x) \,.
    \end{array}
$$
In the above assumptions we have that:
\begin{itemize}
\item $S^{1i}_{jk},S^{2i}_{jk},S^{3i}_{jk},S^{4i}_{jk}$ are singular integral operators associated to singular kernels of fundamental type in $\IR^{n_1}$,
\item the functions $\pp^i_{jk}$, $\qq^i_{jk}$, $\rr^i_{jk} $ belong to $L^1((0,T);L^1(\IR^{n_1}))$, 
\item $\mm^i_{jk}\in L^1((0,T);\IM(\IR^{n_1}))$, 
\item the functions $\gamma^{1,i}_{jk}, \gamma^{2,i}_{jk}, \gamma^{3,i}_{jk}, \gamma^{4,i}_{jk}$ belong to $L^\infty((0,T);L^{q}(\IR^{n_2})) $ for some $q>1$.
\end{itemize} 
\end{itemize}

We have denoted by $L^1((0,T);\IM(\IR^{n_1}))$ the space of all functions $t \mapsto \mu(t,\cdot)$ taking values in the space $\IM(\IR^{n_1})$ of finite signed measures on  $\IR^{n_1}$ such that 
$$
\int_0^T \|\mu (t,\cdot)\|_{\IM(\IR^{n_1})} \, dt<\infty \,.
$$ 

\begin{rmk} 
The assumption on the functions $\gamma^{1,i}_{jk}, \gamma^{2,i}_{jk}, \gamma^{3,i}_{jk}, \gamma^{4,i}_{jk}$ could be relaxed to $L^\infty((0,T);L^{q}_\loc(\IR^{n_2}))$. This would require the use of a localized maximal function.
\end{rmk}

We will additionally assume that
\begin{itemize}
\item[{\bf (R3)}]
$$
b\in L^p_{\loc}([0,T] \times\irn) \qquad \text{ for some $p>1$.}
$$
\end{itemize}

As mentioned in the Introduction, the proof of our main result will exploit an {\em anisotropic} functional (already provisionally introduced in \eqref{anfunc}), which extends the functional \eqref{e:functional} to the regularity setting under investigation. Let $A$ be the constant $N \times N$ matrix
\begin{equation}\label{A}
A = {\rm Diag \,} (\da,\ldots,\da,\db,\ldots,\db) \,.
\end{equation}
$A$ acts on vectors in $\irn$ by a dilation of a factor $\da$ on the first $n_1$ coordinates, and of a factor $\db$ on the last $n_2$ coordinates: $A (x_1,x_2) = (\da x_1,\db x_2)$.

Given $X(t,x_1,x_2)$ and $\widebar{X}(t,x_1,x_2)$  regular Lagrangian flows associated to $b$ and $\bar{b}$ respectively, we denote by $G_\lambda$ and $\wbar{G}_\lambda$ the sublevels of $X$ and $\wbar{X}$ defined as in \eqref{e:sub}. The proof of our main theorem (see Theorem~\ref{fund}) is based on the study of the following anisotropic functional:
\begin{equation}\label{e:anfunctional}
\Phi_{\da,\db}(s)=\int\limits_{B_r\cap G_\la \cap \wbar{G}_\la}
 \log\left(1+ \left| A^{-1}\left[X(s,\x,\y)-\wbar{X}(s,\x,\y)\right] \right| \right) \, dx \,.
\end{equation}

\section{Estimates of anisotropic difference quotients}\label{s:differ}

In this section we first recall the classical estimate for the difference quotients of a $BV$ function, and then recover an analogous ``anisotropic'' version of this result for vector fields in the regularity setting of {\bf (R2)}. This will be a key tool in order to estimate the functional \eqref{e:anfunctional}.

\begin{lemma} \label{prop1}
If $u\in BV(\ird)$, then there exists an $\La^d$-negligible set 
$\mathcal{N}\subset \IR^d$ such that 
$$
|u(x)-u(y)|\leq C_d |x-y|\Big((MDu)(x)+(MDu)(y)\Big)
$$ 
for every $x,y\in \ird \setminus \mathcal{N}$, where $Du$ is the distributional derivative of $u$, represented by a measure.
\end{lemma}

It turns out that an analogous result holds for functions whose derivatives are singular integrals of measures. The following result has been proved in \cite{jhde}. The smooth maximal function in Definition~\ref{d:smoothmax} plays an important role in this estimate. 

\begin{propos}\label{diffquoest}
Let $f\in L^1_{\loc}(\IR^d)$ and assume that for every $j=1,\ldots,d$ we have 
$$
\p_jf =\displaystyle\sum\limits_{k=1}^m R_{jk}g_{jk}  
$$
in the sense of distributions, where $R_{jk}$ are singular integral operators of fundamental type in $\IR^d$ and $g_{jk}\in\IM(\ird)$ for $j=1,\ldots,d$ and $k=1,\ldots,m$, and $R_{jk}g_{jk}$ is defined in the sense of tempered distributions. Then there exists a nonnegative function $V \in M^1(\IR^d)$ and an $\La^d$-negligible set $\mathcal{N}\subset\IR^d$ such that for every $x,y\in\IR^d\setminus\mathcal{N}$ there holds
$$
|f(x)-f(y)| \leq |x-y|\Big(V(x)+V(y)\Big) \,,
$$
where $V$ is given by 
$$
V:=\mathcal{V}(R,g)=\displaystyle\sum\limits_{j=1}^d  
\displaystyle\sum\limits_{k=1}^m M_{\{\Upsilon^{\xi,j}, \,\xi\in {\mathbb S}^{d-1}\}} (R_{jk}g_{jk}) \,.
$$
\end{propos}

In the above proposition $\Upsilon^{\xi,j}$, for $\xi\in {\mathbb S}^{d-1}$ and $j=1,\ldots,d$, is a family of smooth functions explicitly constructed in the course of the proof. 

\begin{rmk}
Theorem~\ref{cancel} implies that the operator $g \mapsto \mathcal{V}(R,g)$ is bounded $L^2 \to L^2$ and $\IM \to M^1$.
\end{rmk}

In the following three subsections we prove similar estimates in the anisotropic context.

\subsection{Split regularity: the isotropic estimate}

Given $\{\gamma^\nu(\x)\}_\nu \subset C_c^\infty(\IR^{n_1})$, $\{\rho^\sigma(\y)\}_\sigma \subset C_c^\infty(\IR^{n_2})$ and $u\in \IS'(\irn)$ we define
\begin{equation}\label{tensormax}
M_{ \{ \gamma^\nu \otimes \rho^\sigma \}} u  (x) = \sup_{\varepsilon>0}\, \sup_{\nu,\sigma}|(\gamma^\nu(\x)\rho^\sigma(\y))_\var\Conv  u(x)|
= \sup_{\varepsilon>0}\, \sup_{\nu,\sigma}\left| \left( \frac{1}{\varepsilon^N} \gamma^\nu \left(\frac{\x}{\varepsilon^{n_1}}\right)\rho^\sigma \left(\frac{\y}{\varepsilon^{n_2}}\right) \right)\Conv  u(x) \right| \,.
\end{equation}
We first of all prove an isotropic estimate in a regularity context related to {\bf (R2)}.

\begin{lemma} \label{lemma0}
Let $f:\irn\rightarrow \IR$ be a function such that for each $j=1,\ldots,N$ we have 
\begin{equation}\label{cond2}
\p_j f = \displaystyle\sum^m_{k=1}( R_{jk}g_{jk}) (\x) \gamma_{jk}(\y) \,,
\end{equation}
where $R_{jk}$ are singular integrals of fundamental type in $\IR^{n_1}$, $g_{jk} \in \IM(\IR^{n_1})$ and $\gamma_{jk}\in L^q(\IR^{n_2})$, for some $q>1$. Then there exists a nonnegative function $V :\irn\rightarrow [0,\infty)$ and an $\La^N$-negligible set $\mathcal{N}\subset\IR^N$ such that for every $x,y\in\IR^N\setminus\mathcal{N}$
$$
|f(x)-f(y)|\leq |x-y|\Big(V(x)+V(y)\Big) \,.
$$
The function $V$ is given by 
\begin{equation}\label{Mtensor}
V:=\IV(R,\gamma,g)=\displaystyle\sum\limits_{j=1}^N  \displaystyle\sum\limits_{k=1}^m M_{\{\Upsilon^{\xi,j}\otimes \bar{\Upsilon}^{\xi,j}\}} (\gamma_{jk} R_{jk}g_{jk} ) \,,
\end{equation}
for suitable smooth compactly supported functions ${\Upsilon}^{\xi,j}$ and $\bar{\Upsilon}^{\xi,j}$, which will be introduced in the proof. 
\end{lemma}
 
\begin{proof}
We adapt the proof of Proposition~\ref{diffquoest} to the current regularity setting. The difficulty is that a smooth maximal function in $\IR^N$ composed with the singular kernel on $\IR^{n_1}$ does not enjoy suitable bounds, and so we use a tensor product of smooth functions, as in \eqref{tensormax}.

Let $w=(w_1,w_2)\in\irn$, and let $\{e_j\}_j$ be the standard basis for $\irn$. We denote $\{w_1\}^j=(w_1,1,\ldots,1)\cdot e_j$ and $\{w_2\}^j=(1,\ldots,1,w_2)\cdot e_j$.
Define the families of functions 
$$
\begin{cases}
\Upsilon^{\xi,j}(w_1)= h^1\left(\frac{\xi_1}{2}-w_1\right) \{w_1\}^j  \\ \\
\bar{\Upsilon}^{\xi,j }(w_2) =h^2\left(\frac{\xi_2}{2}-w_2\right)\{w_2\}^j  \,,
\end{cases}
$$
where $h^i\in C_c^\infty(\IR^{n_i})$ with $\int_{\IR^{n_i}} h^i dx_i= 1$ and $\xi\in {\mathbb S}^{N-1}$. Let $h_r =\frac{1}{r^N}h^1(\frac{\cdot}{r})h^2(\frac{\cdot}{r})$, set $r=|x-y|$, and write
$$
f(x)-f(y)=\int_{\irn} h_r\left(z-\frac{x+y}{2}\right) (f(x)-f(z))dz + \int_{\irn} h_r\left(z-\frac{x+y}{2}\right) (f(z)-f(y))dz \,.
$$
We assume that $f$, $\gamma_{jk}$ and $g_{jk}$ are smooth and compute the following:
$$
\begin{aligned}
& \int\limits_{\irn}  h_r\left(z-\frac{x+y}{2}\right) (f(x)-f(z))dz  \\
& = -\displaystyle\sum\limits_{j=1}^N \int\limits_{\irn} \int_0^1 h_r\left(z-\frac{x+y}{2}\right) \p_jf(x+t(z-x))(z\cdot e_j-x\cdot e_j) \, dtdz \,.
\end{aligned}
$$
After the change of variable $-t(z-x)\mapsto w$ we get
$$
\begin{aligned}
& = \displaystyle\sum\limits_{j=1}^N \int\limits_{\irn} \int_0^1 h_r\left(\frac{x-y}{2}-\frac{w}{t}\right) \p_jf(x-w)\frac{w\cdot e_j }{t^{N+1}} \, dtdw \\
& = r \displaystyle\sum\limits_{j=1}^N \int\limits_{\irn} \int_0^1 \frac{1}{t^N} h_r\left(\frac{x-y}{2}-\frac{w}{t}\right)  \frac{w\cdot e_j}{tr}  \p_jf(x-w) \, dtdw\\
& = r \displaystyle\sum\limits_{j=1}^N \displaystyle\sum^m_{k=1} \int_0^1\Bigg[ \int\limits_{\IR^{n_1}}  \frac{1}{t^{n_1}} h^1_r\left(\frac{x_1-y_1}{2}-\frac{w_1}{t}\right)\left\{\frac{w_1}{tr}\right\}^j   R_{jk}g_{jk}(x_1-w_1) 
 \, dw_1  \\
 &\,\,\,\,\,\,\,\,\,\,\,\,\,\,\,\,\,\,\,\,\,\,\,\,\,\,\,\,\,\, \times \int\limits_{\IR^{n_2}} \frac{1}{t^{n_2}} h^2_r\left(\frac{x_2-y_2}{2}-\frac{w_2}{t}\right)\left\{\frac{w_2}{tr}\right\}^j  \gamma_{jk}(\y-w_2) \, dw_2 \Bigg] \, dt \\
 & = r \displaystyle\sum\limits_{j=1}^N \displaystyle\sum^m_{k=1}  \int_0^1 \left[ \frac{1}{t^{n_1}} h^1_r\left(\frac{x_1-y_1}{2}-\frac{w_1}{t}\right)\left\{\frac{w_1}{tr}\right\}^j \Conv_{w_1}  \,  R_{jk}g_{jk}(w_1) \right](x_1) \\
 & \,\,\,\,\,\,\,\,\,\,\,\,\,\,\,\,\,\,\,\,\,\,\,\,\,\,\,\,\,\,  \times \left[\frac{1}{t^{n_2}} h^2_r\left(\frac{x_2-y_2}{2}-\frac{w_2}{t}\right)\left\{\frac{w_2}{tr}\right\}^j  \Conv_{w_2} \, \gamma_{jk}( w_2) \right](x_2) \, dt \,.
\end{aligned}
$$
Denoting $\Upsilon_\var^{\xi,j}(w_1)=\frac{1}{\var^{n_1}}\Upsilon^{\xi,j}\left(\frac{w_1}{\var}\right)$ and $\bar{\Upsilon}_\var^{\xi,j}(w_2)=\frac{1}{\var^{n_2}}\bar{\Upsilon}^{\xi,j}\left(\frac{w_2}{\var}\right)$, this expression equals 
$$
r \displaystyle\sum\limits_{j=1}^N \displaystyle\sum^m_{k=1}  \int_0^1 [ {\Upsilon}^{\frac{x-y}{|x-y|},j}_{tr} \Conv_1  \,  R_{jk}g_{jk}] \,\,(x_1)\,\,  [ \bar{\Upsilon}^{\frac{x-y}{|x-y|},j}_{tr}\Conv_{2} \, \gamma_{jk}] \,\, (x_2)\,dt , \\
$$
and so
$$
\begin{aligned}
 &\left|\int\limits_{\irn}  h_r\left(z-\frac{x+y}{2}\right) (f(x)-f(z))dz\right| \\
 &\leq |x-y|  \displaystyle\sum\limits_{j=1}^N \displaystyle\sum^m_{k=1}  \int_0^1| [{\Upsilon}^{\frac{x_1-y_1}{|x-y |},j}_{tr} \Conv_1  \,  R_{jk}g_{jk}] \,\,(x_1)\,\, [  \bar{\Upsilon}^{\frac{x_2-y_2}{|x -y |}}_{tr}\Conv_{2} \, \gamma_{jk}  ]\,\, (x_2)|\,dt \\
 &   \leq   |x-y| \displaystyle\sum\limits_{j=1}^N \displaystyle\sum^m_{k=1}  \int_0^1 \sup_{\varepsilon>0} \sup_{\xi  } \,\,  |[{\Upsilon}^{\xi,j}_{\var} \Conv_1  \,  R_{jk}g_{jk}] \,\,(x_1)\,\, [  \bar{\Upsilon}^{\xi,j}_{\var}\Conv_{2} \, \gamma_{jk}]  \,\, (x_2)|\,dt \\
 & =    |x-y|  \displaystyle\sum\limits_{j=1}^N \displaystyle\sum^m_{k=1}  M_{\{ {\Upsilon}^{\xi,j}\otimes  \bar{\Upsilon}^{\xi,j}\}} (\gamma_{jk}R_{jk}g_{jk}) (x) =  |x-y| V(x) \,. 
\end{aligned}
$$  
This proves the statement in the smooth case. By a similar approximation argument as in \cite{jhde}, we conclude this holds for functions of the type in \eqref{cond2}.
\end{proof}

\subsection{Split regularity: the anisotropic estimate}
We now modify Lemma~\ref{lemma0} to obtain an estimate in which distances are measured ``anisotropically'' through the matrix $A$ defined in \eqref{A}. In the next lemma we will use the following notation:
$$
\check{g}_{ij} (x_1) = g_{jk} (\delta_1 x_1) \,, \qquad \check{\gamma}_{ij} (x_2) = \gamma_{ij}(\delta_2 x_2) \,,
$$
where with $g_{jk} (\delta_1 x_1)$ we denote the measure on $\IR^{n_1}$ defined through
$$
\langle g_{jk} (\delta_1 x_1) , \varphi(x_1) \rangle = \delta_1^{-n_1} \langle g_{ij} (y_1) , \varphi( y_1 / \delta_1 ) \rangle \,,
\qquad \varphi \in C^\infty_c(\IR^{n_1}) \,.
$$
Moreover, $R^\da_{jk}$ denotes the singular integral operator in $\IR^{n_1}$ associated to the kernel $K_{jk}^\da$, where
\begin{equation}\label{Ka}
K_{jk}^\da(\x) = \da^{n_1} K_{ij}(\da \x)  \,.  
\end{equation} 

\begin{lemma} \label{lemma1}
Let $f:\irn\rightarrow \IR$ be a function in $L^1_{\loc}(\irn)$ such that for each $j=1,\ldots,N$ we have that $\p_jf$ is as in \eqref{cond2}. Let $A$ be the matrix defined in \eqref{A}. 
Then there exists a nonnegative function $U :\irn\rightarrow [0,\infty)$, such that for $\La^N$-a.e. $x,y\in\irn$, 
$$
|f(x)-f(y)| \leq |A^{-1}[x-y]| \Big( U(x)+U(y) \Big) \,,
$$
where (with the notation above) 
$$
U(x) =\IU(R,\gamma,g)(x)=   \displaystyle\sum\limits_{j=1}^N  \displaystyle\sum\limits_{k=1}^m [M_{\{\Upsilon^{\xi,j}\otimes \bar{\Upsilon}^{\xi,j}\}} (R_{jk}^\da \check{g}_{jk} \check{\gamma}_{jk}A_{jj})](A^{-1}x) \,.
$$
\end{lemma}

\begin{proof}
Define the following rescaled vector field. For each $z \in \irn$,  define
$$
\check{f}(z)=f(Az)\,.
$$
Now $D\check{f}$ is related to $Df$ by the following:  
$$
\p_j\check{f}(z)=\p_jf(Az)A_{jj}= \displaystyle\sum\limits_{k=1}^m \gamma_{jk}(\db z_2)R_{jk}g_{jk}(\da z_1)A_{jj} \,. 
$$
 
We now apply Lemma~\ref{lemma0}. This gives the existence of a function $V \in M^1_{\loc}(\irn)$ to estimate the difference quotient of $\check{f}$:
\begin{equation}\label{mdc}
|\check{f}(z)-\check{f}(w)|\leq |z-w|(V(z)+V(w)) \,,
\end{equation}
with $V$ given by 
\begin{equation} \label{Vfull}
V (z) =\IV(R,\gamma,g)=\displaystyle\sum\limits_{j=1}^N  \displaystyle\sum\limits_{k=1}^m M_{\{\Upsilon^{\xi,j}\otimes \bar{\Upsilon}^{\xi,j}\}} \Big(\gamma_{jk}(\db z_2) R_{jk}g_{jk}(\da z_1) \Big) A_{jj} \,.
\end{equation}
With a change of variable we can verify that
$$
(R_{jk}g_{jk})(\da z_1)=  (R^\da_{jk}\check{g}_{jk} ) (z_1) \,.
$$
Thus we can rewrite \eqref{Vfull} as
\begin{equation}\label{newv}
V(z) =   \displaystyle\sum\limits_{j=1}^N  \displaystyle\sum\limits_{k=1}^m [M_{\{\Upsilon^{\xi,j}\otimes \bar{\Upsilon}^{\xi,j}\}} (R_{jk}^\da \check{g}_{jk} \check{\gamma}_{jk} A_{jj})](z) \,.
\end{equation}
By letting  $U(x)=V(A^{-1}x)$ the proof is concluded.
 \end{proof}

\begin{rmk}
In order to treat the case of a function with gradient given by the singular integral in $\irn$ of a measure, that is 
\begin{equation}\label{cond1}
\p_jf =\displaystyle\sum\limits_{k=1}^m R_{jk}g_{jk} \,,
\end{equation} 
with $R_{jk}$ singular integrals of fundamental type in $\IR^N$ and $g_{jk} \in \IM(\IR^{N})$, one should consider the function
$$
U(x) =\IU(R,g)(x)=   \displaystyle\sum\limits_{j=1}^N  \displaystyle\sum\limits_{k=1}^m [M_{\{\Upsilon^{\xi,j}\}}R^A_{jk}({g}_{jk}(A\cdot))A_{jj})](A^{-1}x) \,,
$$
where $R_{ij}^A$ is the singular integral operator corresponding to the kernel 
$$
K^A_{ij}(x) = |\det A| \, K_{ij}(Ax)
$$
and $A$ is the diagonal matrix defined in \eqref{A}. This would however give a more singular estimate in Lemma~\ref{lemma2part2} below, and would therefore be useless for the proof of Theorem~\ref{fund}.

On the other hand it is possible to treat the case $R_{ij} = \delta$ in \eqref{cond1}, since the Dirac delta ``does not see the dilation''. This would correspond to the case of a vector field $b=(b_1,b_2)$ such that $b_2$ is $BV$ in $x_1$ and $W^{1,1}$ in $x_2$, and $b_1$ is $W^{1,1}$ in both $x_1$ and $x_2$, the situation of \cite{bouarma}. This will be presented in \cite{annathesis}.
\end{rmk}

\subsection{Split regularity: operator bounds}

We finally establish suitable estimates on the norms of the operator defined in Lemma~\ref{lemma1}.
 
\begin{lemma}\label{lemma2part2} 
Let $\IU(R,\gamma,g)$ be as in Lemma~\ref{lemma1}. Then for any $1<p<\infty$ we have
$$
|||\IU(R,\gamma,g)|||_{M^1(\Omega_r)} \leq C_{r,p,m} 
\left( \da \displaystyle\sum^{n_1}_{j=1}\sum^{m}_{k=1} ||\gamma_{jk} ||_{L^p(\IR^{n_2})} ||g_{jk}||_{\IM(\IR^{n_1})} 
+ \db \displaystyle\sum^{N}_{j=n_1+1} \sum^{m}_{k=1} ||\gamma_{jk} ||_{L^p(\IR^{n_2})} ||g_{jk}||_{\IM(\IR^{n_1})} \right) \,,
$$
where $\Omega_r=B_r^1\times B_r^2 \subset \IR^{n_1}\times\IR^{n_2}$, and
$$
||\IU(R,\gamma,g) ||_{L^p(\IR^N)} \leq C_{p} 
\left( \da \displaystyle\sum^{n_1}_{j=1}\sum^{m}_{k=1} ||\gamma_{jk} ||_{L^p(\IR^{n_2})} ||g_{jk}||_{L^p(\IR^{n_1})} 
+ \db \displaystyle\sum^{N}_{j=n_1+1} \sum^{m}_{k=1} ||\gamma_{jk} ||_{L^p(\IR^{n_2})} ||g_{jk}||_{L^p(\IR^{n_1})} \right) \,.
$$
The constants $C_{r,p,m}$ and $C_p$ also depends on the singular integral operators $R_{jk}$ in \eqref{cond2} and on the space dimension. The first constant $C_{r,p,m}$ also depend on the integer $m$ in \eqref{cond2}.
\end{lemma}

\begin{proof}
Let us start with the estimate in $M^1$. We define $\check{B}^1_r = B^1_{r/\delta_1}$, $\check{B}^2_r = B^2_{r/\delta_2}$ and $\check\Omega_r = \check B_r^1 \times \check B_r^2$. Consider first the measure of the superlevels of $U(x)$: changing variable via the linear transformation $z=A^{-1}x$ we obtain
$$
\begin{aligned}
\La^N (\{x \in \Omega_r : |U(x)|>\lambda\} ) = & \; \La^N (\{x\in \Omega_r :| V(A^{-1}x)|>\lambda\}) \\
= & \; \da^{n_1}\db^{n_2}\La^N (\{z\in \check\Omega_r : |V(z)|>\lambda\}) \,,
\end{aligned}
$$ 
where $V$ is as before given by
\begin{equation}\label{Vsplit}
\begin{aligned}
V(z)  = \; & \delta_1 \displaystyle \sum\limits_{j=1}^{n_1}  \displaystyle\sum\limits_{k=1}^m [M_{\{\Upsilon^{\xi,j}\otimes \bar{\Upsilon}^{\xi,j}\}} (R_{jk}^\da \check{g}_{jk} \check{\gamma}_{jk})](z) \\
& + \delta_2 \displaystyle \sum\limits_{j=n_1+1}^{n_1}  \displaystyle\sum\limits_{k=1}^m [M_{\{\Upsilon^{\xi,j}\otimes \bar{\Upsilon}^{\xi,j}\}} (R_{jk}^\da \check{g}_{jk} \check{\gamma}_{jk})](z)
\end{aligned}
\end{equation}
(compare with \eqref{newv} and split the sum for $1 \leq j \leq n_1$ and $n_1+1 \leq j \leq n_1+n_2$).

Remembering that $||| f(x_1,x_2)|||_{M^1_{x_1x_2}} \leq \Big\| ||| f(x_1,x_2) |||_{M^1_{x_1}} \Big\|_{L^1_{x_2}}$ we estimate for fixed $j=1,\ldots,N$ as follows:
$$
\begin{aligned}
\delta_1^{n_1} \delta_2^{n_2} & \left|\left|\left| \displaystyle\sum\limits_{k=1}^m [M_{\{\Upsilon^{\xi,j}\otimes \bar{\Upsilon}^{\xi,j}\}}  (R_{jk}^\da \check{g}_{jk} \check{\gamma}_{jk})](z) \right|\right|\right|_{M^1(\check\Omega_r)} \\ 
\leq  & \; C_m \, \delta_1^{n_1} \delta_2^{n_2} \displaystyle\sum\limits_{k=1}^m
\left|\left|\left| M_{\{\Upsilon^{\xi,j}\}} (R_{jk}^\da \check{g}_{jk} ) \right|\right|\right|_{M^1(\check B^1_r)}
\left\| M_{\{\bar{\Upsilon}^{\xi,j}\}} \check{\gamma}_{jk} \right\|_{L^1(\check B ^2_r)} \\
\leq & \;  C_m [\La^{n_2}(\check B_r^2)]^{1-1/p} \, \delta_1^{n_1} \delta_2^{n_2} \displaystyle\sum\limits_{k=1}^m
\left\| \check{g}_{jk} \right\|_{\IM(\IR^{n_1})} \left\| M_{\{\bar{\Upsilon}^{\xi,j}\}} \check{\gamma}_{jk} \right\|_{L^p(\check B ^2_r)} \\
\leq & \; C_{r,p,m} \, \delta_2^{-n_2 + n_2/p} \, \delta_1^{n_1} \delta_2^{n_2} \displaystyle\sum\limits_{k=1}^m
\left\| \check{g}_{jk} \right\|_{\IM(\IR^{n_1})} \left\|  \check{\gamma}_{jk} \right\|_{L^p(\IR^{n_2})} \\
= & \; C_{r,p,m} \displaystyle\sum\limits_{k=1}^m
\left\|  {g}_{jk} \right\|_{\IM(\IR^{n_1})} \left\| {\gamma}_{jk} \right\|_{L^p(\IR^{n_2})} \,.
\end{aligned}
$$
In the above chain of inequalities we have used the fact that the norm of $R_{jk}^\da$ as singular integral operator coincides with the norm of $R_{jk}$ as singular integral operator. 

Recalling \eqref{Vsplit} we immediately obtain the first inequality claimed in the lemma. The second one follows with a simile argument, using the continuity if the operator 
$$
\check{g}_{jk} \mapsto R_{jk}^\da \check{g}_{jk}
$$
from $L^p(\IR^{n_1})$ into itself.
\end{proof}
  
\section{The fundamental estimate for flows: main theorem and corollaries}\label{s:fund}

Our main theorem is the following:

\begin{theorem}\label{fund}
Let $b$ and $\bar{b}$ be two vector fields satisfying assumption {\bf (R1)}, and assume that $b$ also satisfies assumptions {\bf (R2)} and {\bf(R3)}. Fix $t\in[0,T)$ and let $X$ and $\bar{X}$ be regular Lagrangian flows starting at time $t$ associated to $b$ and $\bar{b}$ respectively, with compressibility constants $L$ and $\bar{L}$.
Then the following holds. For every $\gamma>0$ and $r>0$ and for every $\eta>0$ there exist $\lambda>0$  and $C_{\gamma,r,\eta}>0$ such that
$$\La^n\left(B_r\cap \{|X(s,\cdot)-\bar{X}(s,\cdot)|>\gamma\}\right)\leq C_{\gamma,r,\eta} ||b-\bar{b}||_{L^1((0,T)\times B_\la)}+\eta$$ for all $s\in[t,T]$. 
The constants $\lambda$ and $C_{\gamma,r,\eta}$ also depend on: 
\begin{itemize}
\item The equi-integrability in $L^1((0,T);L^1(\IR^{n_1}))$ of $\pp$, $\qq$, $\rr$, as well as the norm in $L^1((0,T); \IM(\IR^{n_1}))$ of $\mm$ (where $\pp$, $\qq$, $\rr$ and $\mm$ are associated to $b$ as in {\bf (R2)}),
\item The norms of the singular integral operators $S^{\cdot i}_{jk}$, as well as the norms in $L^\infty((0,T); L^q(\IR^{n_2}))$ of $\gamma^{\cdot i}_{jk}$ (associated to $b$ as in {\bf (R2)})),
\item The norm in $L^p((0,T)\times B_\lambda)$ of $b$,
\item The $L^1((0,T);L^1(\irn)) + L^1((0,T);L^\infty(\irn))$ norms of the decompositions of $b$ and $\bar{b}$ as in {\bf (R1)},
\item The compressibility constants $L$ and $\bar{L}$. 
\end{itemize}
\end{theorem}

From this fundamental estimate, the various corollaries regarding the well posedness of the regular Lagrangian flow and of Lagrangian solutions to the continuity and transport equations follow with the same proofs as in Sections~6 and 7 in \cite{jhde}. In particular, we obtain:
\begin{itemize}
\item Uniqueness of the regular Lagrangian flow associated to a vector field satisfying {\bf (R1)}, {\bf (R2)} and {\bf(R3)},
\item Stability (with an explicit rate) for a sequence $X_n$ of regular Lagrangian flows associated to vector fields $b_n$, that converge in $L^1_{\loc}([0,T] \times \IR^N)$ to a vector field satisfying {\bf (R1)}, {\bf (R2)} and {\bf(R3)}, under the assumption that the decompositions of $b_n$ in {\bf (R1)} and the compressibility constants of $X_n$ satisfy uniform bounds,
\item Compactness for a sequence $X_n$ of regular Lagrangian flows associated to vector fields $b_n$ satisfying {\bf (R1)}, {\bf (R2)} and {\bf(R3)} with suitable uniform bounds,
\item Existence of a regular Lagrangian flow associated to a vector field satisfying {\bf (R1)}, {\bf (R2)} and {\bf(R3)} and such that $[{\rm div} \, b]^- \in L^1((0,T);L^\infty(\IR^N))$,
\item If a vector field $b$ satisfies {\bf (R1)}, {\bf (R2)} and {\bf(R3)} and ${\rm div} \, b \in L^1((0,T);L^\infty(\IR^N))$, then there exists a unique forward and backward regular Lagrangian flow associated to $b$, which satisfies the usual group property, and the Jacobian of the flow is well defined,
\item Lagrangian solutions to the continuity and transport equations with a vector field $b$ satisfying {\bf (R1)}, {\bf (R2)} and {\bf(R3)} and ${\rm div} \, b \in L^1((0,T);L^\infty(\IR^N))$ are well defined and stable.
\end{itemize}

\section{Proof of the fundamental estimate (Theorem~\ref{fund})} 

The proof of Theorem~\ref{fund} makes use of the integral functional
$$
\Phi_{\da,\db}(s)=\int\limits_{B_r\cap G_\la \cap \wbar{G}_\la}
 \log\left(1+ \left| A^{-1}\left[X(s,\x,\y)-\wbar{X}(s,\x,\y)\right] \right| \right) \, dx
$$
already defined in \eqref{e:anfunctional}. In the following proof we assume $\da\leq \db$.

In order to improve the readability of the following (many) estimates, we will use the notation ``$\lesssim$'' to denote an estimate up to a constant only depending on absolute constants and on the bounds assumed in Theorem~\ref{fund}, and the notation ``$\lesssim_\lambda$'' to mean that the constant could also depend on the truncation parameter $\lambda$. We will however write explicitly the norm of the measure $\mm$, in order to make the reader aware of its role in the estimates.

\subsubsection*{Step 1: Differentiating $\Phi_{\da,\db}$} We start by differentiating the integral functional with respect to time:
$$
\Phi_{\da,\db}'(s)\leq \int\limits_{B_r\cap G_\la \cap \wbar{G}_\la} \frac{|A^{-1}[b(s,X(s,\x,\y))-\wbar{b}(s,\wbar{X}(s,\x,\y))]|}{1+|A^{-1}[X(s,\x,\y)-\wbar{X}(s,\x,\y)]|} dx \,.
$$
For simplicity, we drop the notation $X (s,\x,\y)$, setting $X(s,\x,\y)=X$ and $\widebar{X}(s,\x,\y) = \widebar{X}$. We estimate
$$
\Phi_{\da,\db}'(s)  \leq       
 \int\limits_{B_r\cap G_\la \cap \wbar{G}_\la} |A^{-1}[b(s,\wbar{X})-\wbar{b}(s,\wbar{X})]|dx + 
 \int\limits_{B_r\cap G_\la \cap \wbar{G}_\la} \frac{|A^{-1}[b(s,X)-{b}(s,\wbar{X})]|}{1+|A^{-1}[X-\wbar{X}]|} dx \,.
$$
After a change in variable along the flow $\wbar{X}$ in the first integral, and noting that $\da\leq\db$, we further obtain
\begin{equation}\label{eq:it}
\begin{aligned}
\Phi_{\da,\db}'(s) \leq \frac{\wbar{L}}{\da} & ||b(s,\cdot)-\wbar{b}(s,\cdot)||_{L^1(B_\lambda)}   \\
&+ \int\limits_{B_r\cap G_\la \cap \wbar{G}_\la} \min\left\{|A^{-1}[b(s,X)-{b}(s,\wbar{X})]|,\frac{|A^{-1}[b(s,X)-{b}(s,\wbar{X})]|}{| A^{-1}[X-\wbar{X}]|}\right\} dx \,.
\end{aligned}
\end{equation}  

\subsubsection*{Step 2: Decomposing the minimum} 
 We consider the second element of the minimum. We have 
$$
A^{-1}[b(s,X)-{b}(s,\wbar{X})]=\left(\frac{b_1(s,X)-b_1(s,\wbar{X})}{\da},\frac{b_2(s,X)-b_2(s,\wbar{X})}{\db}\right) \,,
$$
and therefore
\begin{equation}\label{eq:1}
\frac{|A^{-1}[b(s,X)-{b}(s,\wbar{X})]|}{| A^{-1}[X-\wbar{X}]|}\lesssim \frac{1}{\da}\frac{|b_1(s,X)-b_1(s,\wbar{X})|}{|A^{-1}[X-\wbar{X}]|}+\frac{1}{\db}\frac{|b_2(s,X)-b_2(s,\wbar{X})|}{|A^{-1}[X-\wbar{X}]|} \,.
\end{equation}

\subsubsection*{Step 3: Definition of the functions $U_\pp$, $U_\qq$, $U_\mm$ and $U_\rr$}

We aim at estimating the difference quotients in \eqref{eq:1}. We apply Lemma~\ref{lemma1} and (with a slight extension of the notation) we obtain that
$$
\frac{|b_1(s,x)-b_1(s,\bar{x})|}{|A^{-1}[x-\bar{x}]|} \leq \IU (S^1,S^2,\gamma^1,\gamma^2,\pp,\qq)(x)
+ \IU (S^1,S^2,\gamma^1,\gamma^2,\pp,\qq)(\bar{x}) =: U_{\pp,\qq}(x) + U_{\pp,\qq}(\bar{x})
$$
and
$$
\frac{|b_2(s,x)-b_2(s,\bar{x})|}{|A^{-1}[x-\bar{x}]|} \leq \IU (S^3,S^4,\gamma^3,\gamma^4,\mm,\rr)(x)
+ \IU (S^3,S^4,\gamma^3,\gamma^4,\mm,\rr)(\bar{x}) =: U_{\mm,\rr}(x) + U_{\mm,\rr}(\bar{x})
$$
for a.e.~$x$ and $\bar{x} \in \IR^N$ and $s \in [t,T]$.

It is immediate from the definition of the operator $\IU$ that it is subadditive in its entries. Therefore we can further estimate
$$
U_{\pp,\qq}(x) = \IU (S^1,S^2,\gamma^1,\gamma^2,\pp,\qq)(x)
\leq \IU (S^1,\gamma^1,\pp)(x) + \IU (S^2,\gamma^2,\qq)(x)
=: U_\pp(x) + U_\qq(x)
$$
and
$$
U_{\mm,\rr}(x) = \IU (S^3,S^4,\gamma^3,\gamma^4,\mm,\rr)(x)
\leq \IU (S^3,\gamma^3,\mm)(x) + \IU (S^4,\gamma^4,\rr)(x)
=: U_\mm(x) + U_\rr(x)
$$
for a.e.~$x \in \IR^N$, implying that
\begin{equation}\label{estpq}
\frac{|b_1(s,x)-b_1(s,\bar{x})|}{|A^{-1}[x-\bar{x}]|} \leq U_\pp(x) + U_\qq(x) + U_\pp(\bar{x}) + U_\qq(\bar{x})
\end{equation}
and
\begin{equation}\label{estmr}
\frac{|b_2(s,x)-b_2(s,\bar{x})|}{|A^{-1}[x-\bar{x}]|} \leq U_\mm(x) + U_\rr(x) + U_\mm(\bar{x}) + U_\rr(\bar{x})
\end{equation}
for a.e.~$x$ and $\bar{x} \in \IR^N$ and $s \in [t,T]$.

\subsubsection*{Step 4. Splitting of the quotient}
Let $\Omega = (t,\tau) \times B_r \cap G_\la \cap \wbar{G}_\la \subset \IR^{N+1}$. We return to the estimate in ~\eqref{eq:it} of Step 1. For any $\tau\in[t,T]$ we integrate this expression over $s\in(t,\tau)$, and recall ~\eqref{eq:1} to get 
\begin{equation}
\begin{aligned}
\Phi_{\da,\db} (\tau)  & \lesssim \frac{\wbar{L}}{\da}||b(s,\cdot)-\wbar{b}(s,\cdot)||_{L^1((t,\tau) \times B_\lambda)} \\
& \qquad + \int_\Omega \min\left\{|A^{-1}[b(s,X)-{b}(s,\wbar{X})]| , \frac{1}{\da}\frac{|b_1(s,X)-b_1(s,\wbar{X})|}{|A^{-1}[X-\wbar{X}]|}+   \frac{1}{\db}\frac{|b_2(s,X)-b_2(s,\wbar{X})|}{|A^{-1}[X-\wbar{X}]|} \right\} dxds \\
& = \frac{\wbar{L}}{\da}||b(s,\cdot)-\wbar{b}(s,\cdot)||_{L^1((t,\tau) \times B_\lambda)} + \widetilde{\Phi}_{\da,\db} (\tau) \,.
\end{aligned}
\end{equation}
We analyze the term $\widetilde{\Phi}_{\da,\db} (\tau)$. Using the estimates in \eqref{estpq} and \eqref{estmr} in Step 3, we can write
\begin{equation} \label{eq:decompmdc}
\begin{aligned} 
\widetilde{\Phi}_{\da,\db} (\tau) 
& \lesssim  \int_\Omega \min\left\{|A^{-1}[b(s,X)-{b}(s,\wbar{X})]|, \frac{1}{\da}\frac{|b_1(s,X)-b_1(s,\wbar{X})|}{|A^{-1}[X-\wbar{X}]|} \right\} dx ds \\
& \qquad  +  \int_\Omega \min\left\{|A^{-1}[b(s,X)-{b}(s,\wbar{X})]|, \frac{1}{\db}\frac{|b_2(s,X)-b_2(s,\wbar{X})|}{|A^{-1}[X-\wbar{X}]|} \right\} dx ds \\
& \leq \int_\Omega \min\left\{|A^{-1}[b(s,X)-{b}(s,\wbar{X})]|, \frac{1}{\da}\left( ( {U}_\pp+ {U}_\qq) (s, X)+( {U}_\pp + {U}_\qq) (s, \wbar{X})\right)\right\}  dxds \\
& \qquad + \int_\Omega \min\left\{|A^{-1}[b(s,X)-{b}(s,\wbar{X})]|, \frac{1}{\db}\left( ( {U}_\mm+ {U}_\rr)(s, X)+(( {U}_\mm+ {U}_\rr)) (s, \wbar{X})\right)\right\} dxds \,.
\end{aligned}
\end{equation}

\subsubsection*{Step 5. Decomposition of the functions $U_\pp$, $U_\qq$ and $U_\rr$}
We further decompose the functions $U_\pp$, $U_\qq$ and $U_\rr$ exploiting the equi-integrability of $\pp$, $\qq$ and $\rr$. 

We apply the equi-integrability Lemma~\ref{equi} in $L^1+L^q$, with the same $1<q\leq\infty$ as in the assumption on the functions $\gamma$ in {\bf (R2)}. Given $\var>0$, we find $C_\var>0$, a Borel set $A_\var \subset (0,T) \times \IR^{n_1}$ with finite measure and decompositions
$$
\pp^i_{jk} = (\pp^i_{jk})^1 + (\pp^i_{jk})^2 =: \pp^1 + \pp^2 \,,
$$
$$
\qq^i_{jk} = (\qq^i_{jk})^1 + (\qq^i_{jk})^2 =: \qq^1 + \qq^2 
$$
and
$$
\rr^i_{jk} = (\rr^i_{jk})^1 + (\rr^i_{jk})^2 =: \rr^1 + \rr^2 \,,
$$
so that
$$
\| \pp^1 \|_{L^1((0,T)\times\IR^{n_1})} \leq \var \,, \quad
\| \qq^1 \|_{L^1((0,T)\times\IR^{n_1})} \leq \var \,, \quad
\| \rr^1 \|_{L^1((0,T)\times\IR^{n_1})}  \leq \var \,,
$$
$$
\| \pp^2 \|_{L^q((0,T)\times\IR^{n_1})} \leq C_\var \,, \quad
\| \qq^2 \|_{L^q((0,T)\times\IR^{n_1})} \leq C_\var \,, \quad
\| \rr^2 \|_{L^q((0,T)\times\IR^{n_1})}  \leq C_\var \,,
$$
and
$$
{\rm spt} \, ( \pp^2) \subset A_\var \,, \quad
{\rm spt} \, ( \qq^2) \subset A_\var \,, \quad
{\rm spt} \, ( \rr^2) \subset A_\var \,.
$$

We then decompose the functions $U_\pp$, $U_\qq$ and $U_\rr$ from Step 3 as
$$
U_\pp = \IU(S^1,\gamma^1,\pp) \leq \IU(S^1,\gamma^1,\pp^1) + \IU (S^1,\gamma^1,\pp^2) =: U_\pp^1 + U_\pp^2 \,,
$$
$$
U_\qq = \IU(S^2,\gamma^2,\qq) \leq \IU(S^2,\gamma^2,\qq^1) + \IU (S^2,\gamma^2,\qq^2) =: U_\qq^1 + U_\qq^2 
$$
and
$$
U_\rr = \IU(S^4,\gamma^4,\rr) \leq \IU(S^4,\gamma^4,\rr^1) + \IU (S^4,\gamma^4,\rr^2) =: U_\rr^1 + U_\rr^2 \,.
$$

Applying Lemma~\ref{lemma2part2} to $U^1_\pp$ and $U^2_\pp$ we get
\begin{equation}\label{4} 
\begin{aligned}
||| {U}_{\pp}^1|||_{M^1((0,T) \times B_\lambda)} & \; \lesssim_\lambda \; \da ||\gamma^1||_{L^\infty((0,T); L^q(\IR^{n_2}))} ||\pp^1||_{ L^1((0,T)\times\IR^{n_1}))} \; \lesssim_\lambda \; \da  \var  \,, \\
\| {U}_{\pp}^2 \|_{L^q((0,T) \times B_\lambda)} & \; \lesssim \; \da ||\gamma^1||_{L^\infty((0,T); L^q(\IR^{n_2}))} ||\pp^2||_{L^q((0,T) \times \IR^{n_1} )}  
\; \lesssim \; \da  C_\varepsilon  \,.
\end{aligned}
\end{equation}
We have a similar estimate for $U_{\qq}$ and $U_{\rr}$:
\begin{equation}\label{5}
\begin{aligned}
|||{U}_{\qq}^1|||_{M^1((0,T) \times B_\lambda)} & \lesssim_\lambda  \db\var \,, \qquad
||| {U}_{\rr}^1|||_{M^1((0,T) \times B_\lambda)}   \lesssim_\lambda \db\var \,, \\
\| {U}_{\qq}^2\|_{L^q((0,T) \times B_\lambda)} & \lesssim \db C_\varepsilon \,, \qquad
\| {U}_{\rr}^2\| _{L^q((0,T) \times B_\lambda)}   \lesssim \db C_\varepsilon  \,.
\end{aligned}
\end{equation}

Note that we cannot apply such a decomposition to $ {U}_{\mm}$, since it is defined as the operator $ {\IU}$ acting on a measure rather than integrable function. We only have the bound
\begin{equation}\label{stimam}
||| {U}_{\mm}|||_{M^1((0,T) \times B_\lambda)} \lesssim_\lambda \da ||\mm||_{L^1((0,T);\IM(\IR^{n_1}))} \,.
\end{equation}

We further split the minima according to this decomposition:
\begin{equation}\label{summands}
\begin{aligned}
\widetilde{\Phi}_{\da,\db} (\tau) \lesssim  \int_\Omega \min\left\{|A^{-1}[b(s,X)-{b}(s,\wbar{X})]|, \frac{1}{\db}( {U}_{\mm}(s, X)+ {U}_{\mm}(s, \wbar{X}))\right\} dx ds \\
 +  \int_\Omega \min\left\{|A^{-1}[b(s,X)-{b}(s,\wbar{X})]|, \frac{1}{\db}(   {U}_{\rr}^1(s, X)  + {U}_{\rr}^1  (s, \wbar{X}) )\right\} dx ds 
 \\ +\int_\Omega \min\left\{|A^{-1}[b(s,X)-{b}(s,\wbar{X})]|, \frac{1}{\db}(   {U}_{\rr}^2(s,X)   + {U}_{\rr}^2  (s,\wbar{X}))\right\} dx ds \\
  + \int_\Omega \min \left\{ |A^{-1}[b(s,X) -{b}(s,\wbar{X})]|,  \frac{1}{\da}(   ( {U}_{\pp}^1+  {U}_{\qq}^1)(s,{X})+( {U}_{\pp}^1+ {U}_{\qq}^1)(s, \wbar{X}) )\right\} dx ds \\
  +\int_\Omega \min\left\{|A^{-1}[b(s,X)-{b}(s,\wbar{X})]|, \frac{1}{\da}(   ( {U}_{\pp}^2+ {U}_{\qq}^2)(s, {X})+( {U}_{\pp}^2+ {U}_{\qq}^2)(s, \wbar{X}) )\right\} dx ds \\
   = \int_\Omega\varp_1(s,X,\wbar{X})+  \int_\Omega\varp_2(s,X,\wbar{X})+\int_\Omega\varp_3(s,X,\wbar{X})+\int_\Omega\varp_4(s,X,\wbar{X})+ \int_\Omega\varp_5 (s,X,\wbar{X}) \,.
 \end{aligned} 
 \end{equation}
 
\subsubsection*{Step 6. Estimating the functions $\varphi_j$}
 
Let $\Omega'=(t,\tau)\times B_{\la}\subset\IR^{N+1}$. We estimate the first element of each minima in $L^p$: changing variables along the flows we obtain
\begin{equation} \label{min1}
\| \varphi_j(s,X,\wbar{X})\|_{L^p(\Omega)} \leq 
\frac{L^{1/p} + \wbar{L}^{1/p}}{\da} \|b\|_{L^p(\Omega')} 
\; \lesssim \; \frac{1}{\da}
\end{equation}
for every $j=1,\ldots,5$.

We now consider the second elements of the minima. Let us start with $\varphi_1$. Changing variable along the flows and using \eqref{stimam} we obtain
\begin{equation}\label{eq:phi1}
\begin{aligned}
|||\varphi_1(s,X,\wbar{X})|||_{M^1(\Omega)} 
& \leq \frac{1}{\delta_2}\left|\left|\left| {U}_{\mm}(s, X)+ {U}_{\mm}(s,\wbar{X})\right|\right|\right|_{M^1(\Omega)} \\
& \lesssim  \frac{1}{\db}   |||U_\mm|||_{M^1(\Omega')} 
\; \lesssim_\lambda \frac{\da}{\db} ||\mm||_{L^1((0,T);\IM(\IR^{n_1}))} \,.
\end{aligned}
\end{equation}

Consider $\varp_2$. Using~\eqref{5} we obtain
\begin{equation}\label{eq:phi2}
\begin{aligned}
||| \varp_2(s,X,\wbar{X})|||_{M^1(\Omega)} 
& \leq \frac{1}{\db} ||| {U}_{\rr}^1(s, X) +{U}_{\rr}^1(s, \wbar{X}) |||_{M^1(\Omega)} \\
& \lesssim \frac{1}{\db}|||{U}_{\rr}^1|||_{M^1(\Omega')} \; \lesssim_\lambda \; \var \,. 
\end{aligned}
\end{equation}
 
For $\varp_3$ and $\varp_5$ we neglect the first element of the minimum, since we have directly an estimate on the $L^1(\Omega)$ norm. Using \eqref{5} we obtain
\begin{equation}\label{eq:phi3}
\begin{aligned}
\| \varp_3(s,X,\wbar{X})\|_{L^1(\Omega)} 
& \leq \frac{1}{\db} \|{U}_{\rr}^2(s, X)   +{U}_{\rr}^2  (s, \wbar{X})\|_{L^1(\Omega)} \\
& \lesssim \frac{1}{\db}||   {U}_{\rr}^2||_{L^1(\Omega')} 
\; \lesssim_\lambda \; C_\var \,.
\end{aligned}
\end{equation} 
Similarly, using~\eqref{4} and \eqref{5}, we estimate $\varp_5$ as follows:
\begin{equation}\label{eq:phi5}
\begin{aligned}
||\varp_5(s,X,\wbar{X})||_{L^1(\Omega)} 
& \leq\frac{1}{\da}||   ({U}_{\pp}^2+{U}_{\qq}^2)(s, {X})+({U}_{\pp}^2+{U}_{\qq}^2)(s, \wbar{X})||_{L^1(\Omega)} \\
&\lesssim \frac{1}{\da}||  ( {U}_{\pp}^2+{U}_{\qq}^2)||_{L^1(\Omega')} 
\; \lesssim_\lambda \; \frac{\db}{\da} C_\var \,.
\end{aligned}
\end{equation} 

Finally, using~\eqref{4} and \eqref{5}, we estimate $\varp_4$: 
\begin{equation}\label{eq:phi4}
\begin{aligned}
|||\varp_4 (s,X,\wbar{X}) |||_{M^1(\Omega)} 
& \leq\frac{1}{\da}|||({U}_{\pp}^1+ {U}_{\qq}^1)(s, {X})+({U}_{\pp}^1+{U}_{\qq}^1)(s, \wbar{X})|||_{M^1(\Omega)} \\
&\lesssim \frac{1}{\da} |||({U}_{\pp}^1+{U}_{\qq}^1)|||_{M^1(\Omega')} \\
& \lesssim_\lambda \frac{\da \var + \db \var}{\da} \; \lesssim_\lambda \; \frac{\db}{\da} \var \,.
\end{aligned}
\end{equation}

\subsubsection*{Step 7. Interpolation}
 
We now apply the Interpolation Lemma \ref{interpollemma} to estimate the $L^1(\Omega)$ norms of $\varphi_1$, $\varphi_2$ and $\varphi_4$.
 
Using \eqref{min1} and \eqref{eq:phi1} we obtain
\begin{equation}\label{eq:Phi1}
\| \varp_1(s,X,\wbar{X})\|_{L^1(\Omega)} 
\lesssim_\lambda
\frac{\da}{\db} \| \mm \| \left[ 1 + \log \left( \frac{\db}{\da^2 \|\mm \|} \right) \right] \,.
\end{equation}   

Proceeding similarly and using \eqref{min1}, \eqref{eq:phi2} and \eqref{eq:phi4} we obtain
\begin{equation}\label{eq:phi2,4}
||\varphi_2(s,X,\wbar{X})  ||_{L^1(\Omega)} 
\lesssim_\lambda \var \left[ 1 + \log \left( \frac{1}{\da \var} \right) \right] 
\end{equation}
and
\begin{equation}\label{eq:phi2,4bis}
||\varphi_4(s,X,\wbar{X}) ||_{L^1(\Omega)} 
\lesssim_\lambda \frac{\db}{\da} \var \left[ 1 + \log \left( \frac{1}{\db \var} \right) \right] \,.
\end{equation}

Finally, we sum all the terms in \eqref{summands}. Using ~\eqref{eq:Phi1},~\eqref{eq:phi2,4},~\eqref{eq:phi3},~\eqref{eq:phi2,4bis} and \eqref{eq:phi5}, and setting $\displaystyle\frac{\da}{\db}=\alpha $, we get:
\begin{equation}\label{eq:sphi}
\begin{aligned}
\Phi_{\da,\db}(\tau) & \lesssim_\lambda \frac{1}{\da}||b(s,\cdot)-\wbar{b}(s,\cdot)||_{L^1(B_\lambda\times(t,\tau))} \\
& \qquad + \alpha \| \mm \| \left[ 1 + \log \left( \frac{1}{\da \alpha \|\mm \|} \right) \right]
+ \var \left[ 1 + \log \left( \frac{1}{\da \var} \right) \right]
+ C_\var \\
& \qquad + \frac{\var}{\alpha} \left[ 1 + \log \left( \frac{1}{\db \var} \right) \right]
+ \frac{1}{\alpha} C_\var \,.
\end{aligned}
\end{equation}

\subsubsection*{Step 8. The final estimate}

By definition of $\Phi_{\da,\db}$, given $\gamma>0$ we estimate  
\begin{equation}
\begin{aligned}
\Phi_{\da,\db}(\tau) 
& \geq \int_{B_r\cap\{|X(\tau,x)-\wbar{X}(\tau,x)|>\gamma\}\cap G_\lambda\cap \wbar{G}_\lambda} \log\left(1+\frac{\gamma}{\db}\right) dx \\
& = \log\left(1+\frac{\gamma}{\db}\right)\La^N \Big( B_r\cap\{|X(\tau,x)-\wbar{X}(\tau,x)|>\gamma\}\cap G_\lambda\cap \wbar{G}_\lambda \Big) \,.
\end{aligned}
\end{equation}
This implies that
\begin{equation}\label{eq:lb}
\La^N(B_r\cap\{|X(\tau,x)-\wbar{X}(\tau,x)|>\gamma\})\leq \frac{\Phi_{\da,\db}(\tau)}{ \log\left(1+\frac{\gamma}{\db}\right) }+\La^N(B_r\setminus G_\lambda)+\La^N(B_r\setminus\wbar{G}_{\lambda}) \,.
\end{equation}
Combining~\eqref{eq:sphi} and~\eqref{eq:lb} we obtain
\begin{equation}
\begin{aligned}
\La^N  (B_r\cap & \{|X(\tau,x)-\wbar{X}(\tau,x)| >\gamma\}) \\
& \leq \; C_\lambda \, \Big\{ \frac{\frac{1}{\da} \| b -\wbar{b}\|_{L^1}}{\log\left(1+\frac{\gamma}{\db}\right) }
+ \frac{\alpha \| \mm \| \left[ 1 + \log \left( \frac{1}{\da \alpha \|\mm \|} \right) \right]}{\log\left(1+\frac{\gamma}{\db}\right)}
+ \frac{\var \left[ 1 + \log \left( \frac{1}{\da \var} \right) \right]}{\log\left(1+\frac{\gamma}{\db}\right)} \\
& \qquad \qquad \qquad \qquad + \frac{\frac{\var}{\alpha} \left[ 1 + \log \left( \frac{1}{\db \var} \right) \right]}{\log\left(1+\frac{\gamma}{\db}\right) }
+ \frac{\frac{1}{\alpha} C_\var}{\log\left(1+\frac{\gamma}{\db}\right) } 
+ \frac{C_\var} {\log\left(1+\frac{\gamma}{\db}\right) } \Big\} \\
& \qquad \qquad + \La^N(B_r\setminus G_\lambda)+ \La^N(B_r\setminus \wbar{G}_\lambda) \\
& =: 1)+2)+3)+4)+5)+6)+7)+8) \,.
     \end{aligned} 
\end{equation}
 
Fix $\eta>0$. By Lemma \ref{estsuper}, we can choose $\lambda>0$ large enough so that $7)+8)\leq 2\eta/7$. Choose $\alpha$ small enough so that $2)\leq \eta/7$. Then choose $\var <\alpha ^2$ small enough so that $3)+4)\leq 2\eta/7$, since these terms are uniformly bounded as $\da,\db\rightarrow 0$ and for all $\var>0$. 

Now $\lambda$ and $\var$ (and therefore $C_\var$) are fixed. Also $\alpha$ is fixed, but $\da$ and $\db$ are free to be chosen so long as the ratio equals $\alpha$. Hence, we now choose $\db$ small enough, in particular depending on $C_\var$, so that $5)+6)\leq 2\eta/7$. This fixes all parameters. 

Setting
$$
C_{\gamma,r,\eta}=\frac{C_\lambda}{\da\log(1+\frac{\gamma}{\db})}
$$
we have proven our statement. \hfill \qed

\end{document}